\documentclass[11pt]{amsart}
\usepackage{amssymb,amsthm,amsmath,epsfig,latexsym}
\usepackage{calc,times,verbatim}
\usepackage{geometry} 
\begin{document}

\newcommand{\mmbox}[1]{\mbox{${#1}$}}
\newcommand{\proj}[1]{\mmbox{{\mathbb P}^{#1}}}
\newcommand{\Cr}{C^r(\Delta)}
\newcommand{\CR}{C^r(\hat\Delta)}
\newcommand{\affine}[1]{\mmbox{{\mathbb A}^{#1}}}
\newcommand{\Ann}[1]{\mmbox{{\rm Ann}({#1})}}
\newcommand{\caps}[3]{\mmbox{{#1}_{#2} \cap \ldots \cap {#1}_{#3}}}
\newcommand{\Proj}{{\mathbb P}}
\newcommand{\N}{{\mathbb N}}
\newcommand{\Z}{{\mathbb Z}}
\newcommand{\R}{{\mathbb R}}
\newcommand{\A}{{\mathcal{A}}}
\newcommand{\Tor}{\mathop{\rm Tor}\nolimits}
\newcommand{\Ext}{\mathop{\rm Ext}\nolimits}
\newcommand{\Hom}{\mathop{\rm Hom}\nolimits}
\newcommand{\im}{\mathop{\rm Im}\nolimits}
\newcommand{\rank}{\mathop{\rm rank}\nolimits}
\newcommand{\supp}{\mathop{\rm supp}\nolimits}
\newcommand{\arrow}[1]{\stackrel{#1}{\longrightarrow}}
\newcommand{\CB}{Cayley-Bacharach}
\newcommand{\coker}{\mathop{\rm coker}\nolimits}
\sloppy
\theoremstyle{plain}

\newtheorem*{thm*}{Theorem}
\newtheorem{defn0}{Definition}[section]
\newtheorem{prop0}[defn0]{Proposition}
\newtheorem{quest0}[defn0]{Question}
\newtheorem{thm0}[defn0]{Theorem}
\newtheorem{lem0}[defn0]{Lemma}
\newtheorem{corollary0}[defn0]{Corollary}
\newtheorem{example0}[defn0]{Example}
\newtheorem{remark0}[defn0]{Remark}
\newtheorem{conj0}[defn0]{Conjecture}

\newenvironment{defn}{\begin{defn0}}{\end{defn0}}
\newenvironment{conj}{\begin{conj0}}{\end{conj0}}
\newenvironment{prop}{\begin{prop0}}{\end{prop0}}
\newenvironment{quest}{\begin{quest0}}{\end{quest0}}
\newenvironment{thm}{\begin{thm0}}{\end{thm0}}
\newenvironment{lem}{\begin{lem0}}{\end{lem0}}
\newenvironment{cor}{\begin{corollary0}}{\end{corollary0}}
\newenvironment{exm}{\begin{example0}\rm}{\end{example0}}
\newenvironment{rem}{\begin{remark0}\rm}{\end{remark0}}

\newcommand{\defref}[1]{Definition~\ref{#1}}
\newcommand{\conjref}[1]{Conjecture~\ref{#1}}
\newcommand{\propref}[1]{Proposition~\ref{#1}}
\newcommand{\thmref}[1]{Theorem~\ref{#1}}
\newcommand{\lemref}[1]{Lemma~\ref{#1}}
\newcommand{\corref}[1]{Corollary~\ref{#1}}
\newcommand{\exref}[1]{Example~\ref{#1}}
\newcommand{\secref}[1]{Section~\ref{#1}}
\newcommand{\remref}[1]{Remark~\ref{#1}}
\newcommand{\questref}[1]{Question~\ref{#1}}

\newcommand{\std}{Gr\"{o}bner}
\newcommand{\jq}{J_{Q}}

\title{On the containment problem for fat points ideals}
\author{\c{S}tefan O. Toh\v{a}neanu and Yu Xie}

 \keywords{symbolic powers, fat points, very general points, line arrangements. \\
\indent Tohaneanu's address: Department of Mathematics, University of Idaho, Moscow, Idaho 83844-1103, USA, Email: tohaneanu@uidaho.edu.\\
\indent Xie's address: Department of Mathematics, Widener University, Chester, Pennsylvania 19013, USA, Email: yxie@widener.edu.}

\begin{abstract}
In this note we show that Harbourne's conjecture is true for symbolic powers of ideals of points, we check that the stable version of this conjecture is valid for ideals of very general points (resp. generic points) in $\mathbb P_{\mathbb K}^N$ (resp. $\mathbb P_{\mathbb K(\underline{z})}^N$). We also show that this conjecture and the Harbourne-Huneke conjecture are true for a class of ideals $I$ defining fat points obtained from line arrangements in $\mathbb P_{\mathbb K}^2$.
\end{abstract}

\maketitle

\section{Introduction}

Let $R:=\mathbb K[x_0,  \ldots, x_N]$ be a polynomial ring over a field $\mathbb K$ of characteristic 0. Let $I=\cap_{i=1}^s I(P_i)$ be the defining  ideal of a set of $s$ (distinct) points $P_1, \ldots, P_s$ in $\mathbb{P}_{\mathbb K}^N$, where $I(P_i)$ is the ideal generated by all homogeneous polynomials that vanish at $P_i$.

The $m$-th symbolic power of $I$ is defined as the ideal of fat points $I^{(m)}=\cap_{i=1}^s I(P_i)^m$. It is clear that $I^m\subseteq I^{(m)}$, but  $I^{(m)}$ is not contained in $I^m$ in general. The containment problem is to determine  all the values of $m$ and $r$ for which $I^{(m)}\subseteq I^r$ holds. A fundamental result of Ein-Lazarsfeld-Smith \cite{ELS01} and Hochster-Huneke \cite{HH02} proved  that $I^{(Nm)}\subseteq I^m$ for any $m\geq 1$; the theorem is valid for any homogeneous ideal $I$, and for a field of any characteristic. Later, after positive answers in a large class of examples, and a multitude of other experiments, this containment result can be improved into the following conjecture, often known as Harbourne's conjecture:
\begin{conj} \label{conj1}(\cite[Conjecture 8.4.2]{BRHKKSS09}, \cite[Problem 1.5]{SzSz})\,
Let $I=\cap_{i=1}^sI(P_i)$ be the ideal of a finite set of $s$ points in $\mathbb{P}_{\mathbb K}^N$. Then for all $m>0$,
$$
I^{(Nm-N+1)}\subseteq I^m.
$$
\end{conj}

Harbourne's conjecture has been verified for some classes of ideals, including radical ideals of finite sets of general points when $N=2, 3$ (See \cite[Theorem 4.1]{BH09} and \cite[Theorem 3]{D15}), monomial ideals (\cite[Example 8.4.5]{BRHKKSS09}), radical ideals of a finite set of points and $m$ is a power of ${\rm char}(k)>0$ (\cite[Example 8.4.4]{BRHKKSS09}),  ideals in regular rings $R$  of prime characteristic $p$ with $R/I$ is F-pure, or in regular rings $R$ essentially of finite type over a field of characteristic $0$ with $R/I$ is of dense F-pure (\cite[Theorems 3.3 and 5.3]{GH17}). Despite this success, for specific values of $m$, interesting counterexamples have been found: see \cite[Conterexamples 4.1-4.5]{SzSz} for counterexamples in characteristic 0, and $m=3, N=2$; see \cite[Conterexamples 3.9, 3.10, 4.8, 4.9]{SzSz} for counterexamples in positive characteristic, some in higher dimensions. Later we will take a brief yet closer look at the first counterexample from this list, which is due to Dumnicki, Szemberg, and Tutaj-Gasi${\rm \acute{n}}$ska \cite{DST13}. This counterexample, combined with the positive answer to Harbourne's conjecture when $I$ is replaced by any symbolic power $J^{(t)},t\geq 2$ (see Proposition \ref{1}), determined us to investigate in Section~3 if the ideal $I$ defining the fat-points singularity locus of a line arrangement in $\mathbb P_{\mathbb K}^2$, verifies Harbourne's conjecture, as well as other conjectures related to containment problems (see Remark \ref{remark2}).

Recently, Grifo \cite{G18} established some sufficient conditions to guarantee a stable version of Harbourne's conjecture that the containment $I^{(Nm-N+1)}\subseteq I^m$ holds for all sufficiently large values of $m$. In Theorem \ref{theorem} and Remark \ref{generic}, we check the stable version for finite sets of very general points (resp. generic points) in $\mathbb P_{\mathbb K}^N$ (resp. $\mathbb P_{\mathbb K(\underline{z})}^N$).

As it is the case of these notes, when dealing with homogeneous ideals $I\subset R$ defining 0-dimensional subschemes of $\mathbb P_{\mathbb K}^N$, one uses the following {\em Postulation Containment Criterion}: if $r\cdot{\rm reg}(I)\leq\alpha(I^{(m)})$, then $I^{(m)}\subset I^r$, where ${\rm reg}(I)$ is the Castelnuovo-Mumford regularity, and for any homogeneous ideal $J$, $\alpha(J)$ is the least degree of a nonzero element of $J$ (see \cite[Proposition 3.6]{SzSz}, or \cite{BoCoHa,BH09}).

Another invariant that plays a crucial role in the containment problem is the resurgence. {\em The resurgence} $\rho(I)$ is defined as the supremum of all ratios $m/r$ such that $I^{(m)} \not\subseteq I^r$. By this definition for any pair $m$ and $r$ of positive integers, if $m/r>\rho(I)$, then $I^{(m)}\subseteq I^r$. By the result of Ein-Lazarsfeld-Smith and Hochster-Huneke that $I^{(Nm)}\subseteq I^m$ for any $m\geq 1$, one can see that $\rho(I)\leq N$. On the other hand, a lower bound of $\rho(I)$ is obtained in \cite[Lemma 2.3.2]{BH09}  $$\rho(I)\geq \frac{\alpha(I)}{\widehat{\alpha}(I)}\geq 1,$$ where $\displaystyle\widehat{\alpha}(I)={\rm lim}_{m\to \infty}\frac{\alpha\left(I^{(m)}\right)}{m}$. In Proposition \ref{1}, we provide sharper upper bounds for the resurgences of symbolic powers of ideals of points and obtain their asymptotic behavior.

\bigskip

\section{Containment problem for (fat) points ideals}

\bigskip

Let $I\subset R=\mathbb K[x_0,\ldots,x_N]$ be the ideal of $s$ points in $\mathbb P_{\mathbb K}^N$. So $\displaystyle I=\cap_{i=1}^s I(P_i)$, where $I(P_i)$ is the ideal of the point $P_i$. Let $t\geq 2$ be an integer. Then $$I^{(t)}=\bigcap_{i=1}^s I(P_i)^t$$ is the defining ideal of the (homogeneous) fat-points scheme $tP_1+\cdots+tP_s$; each of the points is ``fattened'' with multiplicity $t$. This is still a saturated ideal, but it is not reduced. As we will show below, Harbourne's conjecture is true for $I^{(t)}$, for any $t\geq 2$.

\begin{prop} \label{1}
Let $I\subset R$ be the ideal of $s$ points in $\mathbb P_{\mathbb K}^N$. Then

\begin{itemize}
  \item[(1)] $\left(I^{(t)}\right)^{(Nm-N+1)}\subseteq \left(I^{(t)}\right)^m \mbox{ for any } m\geq 1$ and $t\geq 2$.
  \item[(2)] $\rho\left(I^{(t)}\right)\leq \frac{t+N-1}{t}$ for any $t\geq 1$.
  \item[(3)] $\displaystyle \lim_{t \to \infty} \rho \left(I^{(t)}\right)=1$.
\end{itemize}
\end{prop}
\begin{proof} We first prove Part (1).  It is clear the result holds for $m=1$. So we may assume $m\geq 2$. By \cite[Theorem 4.4]{HH02}, we have $I^{(m(t+N-1))}\subseteq \left(I^{(t)}\right)^m$, hence we only need to show $\left(I^{(t)}\right)^{(Nm-N+1)}=I^{(t(Nm-N+1))}\subseteq I^{(m(t+N-1))}$. So we need $t(Nm-N+1)\geq m(t+N-1)$, i.e., $t(N-1)(m-1)\geq m(N-1)$,\, $t\geq m/m-1$, which holds since $m\geq 2$ and $t\geq 2$.

\bigskip

For Part (2) and  Part (3),  again by \cite[Theorem 4.4]{HH02}, we have $I^{(r (t+N-1))}\subseteq \left(I^{(t)}\right)^{r}$ for any positive integers $r$ and $t$.
Let $m$ and $r$ be two positive integers such that $m/r>\frac{t+N-1}{t}$, we will show that $\left(I^{(t)}\right)^{(m)}\subset \left(I^{(t)}\right)^r$.
Let $(t+N-1)r=tq+h$, where $1\leq h\leq t$. Since $m$ is an integer and $m>\frac{(t+N-1)r}{t}$, we have $m\geq q+1$. Hence
$$
\left(I^{(t)}\right)^{(m)}\subseteq \left(I^{(t)}\right)^{(q+1)}=I^{(t(q+1))}\subset I^{(tq+h)}= I^{((t+N-1)r)}\subset \left(I^{(t)}\right)^r.
$$
Hence
$$ \rho\left(I^{(t)}\right)\leq \frac{t+N-1}{t} \mbox{ for }  t\geq 1,$$
and consequently, $\displaystyle \lim_{t \to \infty} \rho \left(I^{(t)}\right)=1$.
\end{proof}

\bigskip

For any nonzero vector $\underline{\lambda}=(\lambda_{ij}) \in \mathbb{A}_{\mathbb K}^{s(N+1)}$, where $1\leq i\leq s, 0\leq j\leq N$,  we define the set of points $\{P_1,\ldots, P_s\}\subseteq \mathbb P_{\mathbb K}^N$ as the points  $P_i=P_i(\underline{\lambda})=[\lambda_{i0}: \lambda_{i1}: \ldots: \lambda_{iN}]\in \mathbb P_{\mathbb K}^N$.  One says $\{P_1, \ldots, P_s\}$ is a set of $s$ general points in $\mathbb{P}_{\mathbb K}^N$ if there is a dense Zariski-open subset $W$ of $\mathbb{A}_{\mathbb K}^{s(N+1)}$ such that $\underline{\lambda}=(\lambda_{ij}) \in W$. Similarly, one says $\{P_1, \ldots, P_s\}$ is a set of $s$ very general points in $\mathbb{P}_{\mathbb K}^N$ if $\underline{\lambda}=(\lambda_{ij}) \in W$, where $W=\bigcap \limits_{i=1}^{\infty} U_i$ and the $U_i$ are dense Zariski-open subsets of $\mathbb{A}_{\mathbb K}^{s(N+1)}$ (when $\mathbb K$ is uncountable, then $W$ is actually a dense (not necessarily open) subset).

For an ideal $I$ of points, $I^{(m)}$ is the saturation of $I^{m}$ with respect to the irrelevant ideal $M=(x_0, x_1, \ldots, x_N)$. By definition, $\left(I^m\right)_t=\left(I^{(m)}\right)_t$ for $t\geq {\rm satdeg} (I^m)$, where ${\rm satdeg} (I^m)$ is the saturation degree of $I^m$ (here $J_t$ denotes the set of all elements of degree $t$ in $J$). We also have the inequalities
$$
{\rm satdeg}(I^m)\leq {\rm reg}(I^m)\leq m\cdot {\rm reg}(I).
$$
(The first inequality holds for any homogeneous ideal \cite[Lemma 1.5]{DS02} and the second follows from \cite{C97}  for the ideals with Krull dimension at most $1$).
Furthermore, if $I$ is the defining ideal of $s$ (very) general points, we also have $r \leq \alpha(I)\leq r+1$  and ${\rm reg}(I)=r+1\leq \alpha(I)+1$, where $r$ is the integer such that
$$
\binom{r-1+N}{N}<s\leq \binom{r+N}{N}
$$
see for example \cite[Page 1177]{BH10}.

In the following, we will show that Harbourne's conjecture is valid for ideals of very general points in $\mathbb P_{\mathbb K}^N$ for sufficiently large values $m$.

\bigskip

\begin{thm}\label{theorem}
   Let $I=\cap_{i=1}^s I(P_i)$ be the ideal of $s$ very general points in $\mathbb{P}_{\mathbb K}^N$. Set $\beta=\frac{2(N-1)\left(\alpha(I)+N-1\right)}{(N-2)N}+~1$ if $N\geq 4$, or $\beta=1$ if $N=2, 3$. Then for $m\geq \beta$, one has
   $$
   I^{(Nm-N+1)}\subseteq I^m.
   $$
\end{thm}
\begin{proof}
The theorem is true if $N=2$ by \cite[Theorem 4.1]{BH09}, and if $N=3$ by \cite[Theorem 3]{D15}. So we may assume $N\geq 4$.

\bigskip

First we prove the result if $m=2r$ for some $r\geq 2$, i.e.,  $I^{(2rN-N+1)}\subseteq I^{2r}$.
Since $$I^{(2rN-N+1)}\subseteq I^{(2rN-2(N-1))}=\left(I^{(2)}\right)^{(rN-N+1)},$$
by Proposition \ref{1}, this is included in $\left(I^{(2)}\right)^r$.

\bigskip

\noindent {\em Claim}: $[\left(I^{(2)}\right)^r]_t=[I^{2r}]_t$ for $t\geq 2r(\alpha(I)+1)$.

\bigskip

{\em Proof of Claim}: Since $I^2\subseteq I^{(2)}$, it is clear that $[I^{2r}]_t\subseteq [\left(I^{(2)}\right)^r]_t$, for any $t$. On the other hand,  since $$2r (\alpha(I)+1)\geq 2r \, {\rm reg}(I)\geq {\rm reg}(I^{2r})\geq  {\rm satdeg}(I^{2r}),$$
so if $t\geq  2r(\alpha(I)+1)$, then $$\left[\left(I^{(2)}\right)^r\right]_t\subseteq \left[I^{(2r)}\right]_t=\left[I^{2r}\right]_t.$$

We need to show $\alpha\left(I^{(2rN-N+1)}\right)\geq 2r(\alpha(I)+1)$. Indeed, we can show $\alpha\left(I^{(2rN-2N+2)}\right)\geq  2r(\alpha(I)+1)$. Suppose not, then we have $\alpha\left(I^{(2rN-2N+2)}\right) < 2r(\alpha(I)+1)$.
Since $I$ is the ideal of $s$ very general points in $\mathbb{P}_{\mathbb{K}}^N$, by \cite{DT17} and \cite{FMX18}, $\widehat{\alpha}(I)\geq \frac{\alpha(I)+N-1}{N}$. Hence
$$
\frac{\alpha(I)+N-1}{N}\leq \widehat{\alpha}(I)\leq \frac{\alpha\left(I^{(2rN-2N+2)}\right)}{2rN-2N+2}< \frac{2r(\alpha(I)+1)}{2rN-2(N-1)}.
$$
By computation, we have $\left(\alpha(I)+N-1\right)\left(2rN-2(N-1)\right)< 2rN(\alpha(I)+1)$. Then
$$
2rN\left(N-2\right)< 2(N-1)\left(\alpha(I)+N-1\right)
$$
which implies $m=2r< \frac{2(N-1)\left(\alpha(I)+N-1\right)}{(N-2)N}$, a contradiction.

Now assume $m=2r+1$ for some $r\geq 2$. Then we have $2r\geq \frac{2(N-1)\left(\alpha(I)+N-1\right)}{(N-2)N}$. Hence by the above proof, we have
$I^{(2rN-2(N-1))}\subseteq I^{2r}$.
By \cite[Theorem 1.1]{J14}, for any $\ell\geq 1$ and $a_1, \ldots, a_{\ell}\geq 0$, we have
$$
I^{(N\ell+a_1+\cdots+a_{\ell})}\subseteq I^{(a_1+1)}\cdots I^{(a_{\ell}+1)}.
$$
Hence
$$
I^{((2r+1)N-N+1)}=I^{\left(2N+(2rN-2N+1)+0\right)}\subseteq I^{(2rN-2N+2)}\cdot I\subseteq I^{2r}\cdot I=I^{2r+1}.
$$
\end{proof}

\medskip

\begin{rem} \label{generic} Let $S=\mathbb{K}(\underline{z})[x_0, \ldots, x_N]$, where $\mathbb{K}\subset \mathbb{K}(\underline{z})$ is a pure transcendental extension of fields by adjoining $s(N+1)$ variables $\underline{z}=\left(z_{ij}\right)_{1\leq i\leq s, \, 0\leq j\leq N}$.   A set of $s$ generic points $P_1, \ldots, P_s$ consists of points $P_i = [z_{i0} :  z_{i1}: \ldots : z_{iN}]$. By a similar proof as in Theorem \ref{theorem}, one can show that a stable version of Harbourne's conjecture also holds for the defining ideal $H=\cap_{i=1}^s I(P_i)$ of $s$ generic points.
\end{rem}

\medskip

\section{The containment problem for fat points derived from line arrangements}

\medskip

Let $\A$ be an arrangement of $n$ lines in $\mathbb P_{\mathbb{K}}^2$. Suppose we fixed $\ell_1,\ldots,\ell_n\in R:=\mathbb K[x,y,z]$, which are the defining linear forms of $\A$, and suppose ${\rm ht}(\langle\ell_1,\ldots,\ell_n\rangle)=3$ (i.e., the rank of $\A$ is 3).
Let $$I:=I_{n-1}(\ell_1\cdots\ell_n)=\langle \ell_2\ell_3\cdots\ell_n,\ell_1\ell_3\cdots\ell_n,\ldots,\ell_1\ell_2\cdots\ell_{n-1}\rangle,$$ be the ideal generated by all $(n-1)-$fold products of the linear forms defining $\A$.
By \cite[Lemmas 3.1 and 3.2]{Sc}, the ideal $I$ has the following properties:

\begin{itemize}
  \item[(i)] $I$ has the primary decomposition $$I=I(P_1)^{n_1-1}\cap\cdots\cap I(P_s)^{n_s-1},$$ where $P_1,\ldots,P_s$ is the singular locus (i.e., the intersection points) of the line arrangement $\A$, and for $j=1,\ldots,s$, $I(P_j)$ is the ideal of the point $P_j$, and $n_j$ is the number of lines of $\A$ passing through $P_j$.
  \item[(ii)] The ideal $I$ has graded minimal free resolution:
$$0\longrightarrow R(-n)^{n-1}\longrightarrow R(-(n-1))^n\longrightarrow I\longrightarrow 0.$$ One consequence is that $\alpha(I)={\rm reg}(I)=n-1$.
\end{itemize}

\begin{rem}\label{remark} The counterexample to the containment $J^{(3)}\subset J^2$ due to \cite{DST13} uses the reduced Jacobian scheme of the line arrangement $\A$ in $\mathbb P^2$ with defining polynomial $(x^3-y^3)(y^3-z^3)(z^3-x^3)$. This means that $J$ is the defining ideal of the 12 singular points of $\A$. But $I_8(\A)$, by property (i) above, equals $I(P_1)^2\cap\cdots\cap I(P_{12})^2=J^{(2)}$, and from Proposition \ref{1}, we have $I_8(\A)^{(3)}\subset I_8(\A)^2$. This is not just a simple coincidence; in Corollary \ref{corollary} below we will show the containment problem for this special class of fat points ideals derived from line arrangements.
\end{rem}

\medskip

\begin{prop} \label{theorem2} For $I$ as above and for any $r\geq 1$ we have $$\alpha\left(I^{(2r-1)}\right)\geq rn-1 \mbox{ and }\alpha\left(I^{(2r)}\right)\geq rn.$$ Furthermore, if the lines of $\A$ intersect generically, then the above inequalities become equalities.
\end{prop}
\begin{proof} Let $F\in I^{(2r-1)}$ be of degree $rn-2$. We will show that $F$ must be the zero polynomial (which has any degree).

We have $I=\cap_{j=1}^sI_j^{m_j}$, where we denote $I_j:=I(P_j)$, and $m_j:=n_j-1\geq 1$. Then, by definition, $I^{(2r-1)}=\cap_{j=1}^sI_j^{(2r-1)m_j}$.

Denote by $D_{r-1}F$ to be any arbitrary partial derivative of order $r-1$ of $F$. Then $\deg(D_{r-1}F)=rn-2-r+1=rn-r-1$. Also, since $(m_j-1)(r-1)\geq 0$, then $(2r-1)m_j-(r-1)\geq rm_j$. So, for all $j=1,\ldots,s$, one has $D_{r-1}F\in I_j^{(2r-1)m_j-r+1}\subseteq I^{rm_j}$, and therefore $$D_{r-1}F\in I^{(r)}.$$

Let $V(\ell_i)$ be any arbitrary line of $\A$. Suppose $P_1,\ldots,P_p$ are all the intersection points of $\A$ lying on this line $V(\ell_i)$. Suppose $\gcd(\ell_i,D_{r-1}F)=1$. We have that $D_{r-1}F$ is a homogeneous polynomial of degree $rn-r-1$, and it vanishes of order $rm_k$ at each of the points $P_k, k=1\ldots,p$. By B\'{e}zout's Theorem, we have

$$\deg(\ell_i)\cdot\deg(D_{r-1}F)=rn-r-1\geq \sum_{k=1}^p rm_k.$$ But, since $\A$ is a line arrangement in $\mathbb P^2$, one has $$\sum_{k=1}^p(n_k-1)=n-1.$$ Everything put together gives $$r(n-1)-1\geq r(n-1),$$ which is an obvious contradiction.

So $\ell_i|D_{r-1}F$, and since $\ell_i$ was arbitrary, we have $$(\ell_1\cdots\ell_n)|D_{r-1}F,$$ for any partial derivative of order $r-1$ of $F$.\footnote{This simple trick of using B\'{e}zout's Theorem to obtain linear forms dividing an element in the symbolic power we learned it from the proof of \cite[Corollary 3.9]{HaHu}}

\bigskip

\noindent {\em Claim}: Denote $A:=\ell_1\cdots\ell_n$. Let $1\leq j\leq r$. Then, if $A^j$ divides all partial derivative of order $r-j$ of $F$, then $A^{j+1}$ divides all partial derivative of order $r-j-1$ of $F$.

{\em Proof of Claim}: From Euler's relation, any partial derivative of order $r-j-1$ of $F$ can be written as combination of partial derivatives of order $r-j$ of $F$, and hence it is divisible by $A^j$.

Let $D$ be any such partial derivative of order $r-j-1$ of $F$. Therefore $$D=A^jG, \mbox{ for some }G\in R.$$

Suppose, for $i=1,\ldots,n$, $\ell_i=a_ix+b_iy+c_iz, a_i,b_i,c_i\in\mathbb K$. Then, the partial derivatives of $D$ with respect to $x, y,$ and $z$ are:

\begin{eqnarray}
D_x&=&jA^{j-1}[a_1(\ell_2\cdots\ell_n)+\cdots+a_n(\ell_1\cdots\ell_{n-1})]G+A^jG_x,\nonumber\\
D_y&=&jA^{j-1}[b_1(\ell_2\cdots\ell_n)+\cdots+b_n(\ell_1\cdots\ell_{n-1})]G+A^jG_y,\nonumber\\
D_z&=&jA^{j-1}[c_1(\ell_2\cdots\ell_n)+\cdots+c_n(\ell_1\cdots\ell_{n-1})]G+A^jG_z.\nonumber
\end{eqnarray}

Each $D_x$, $D_y$, $D_z$ is divisible by $A^j$, since they are partial derivatives of order $r-j$ of $F$.

If there exists $i_0\in\{1,\ldots,n\}$ such that $\ell_{i_0}\nmid G$, then, since, after simplifying by $A^{j-1}$, we have $\ell_{i_0}$ dividing each of $\displaystyle \frac{D_x}{A^{j-1}}-AG_x, \frac{D_y}{A^{j-1}}-AG_y, \frac{D_z}{A^{j-1}}-AG_z$, we have $a_{i_0}=b_{i_0}=c_{i_0}=0$, and hence $\ell_{i_0}=0$; a contradiction. Therefore, $(\ell_1\cdots\ell_n)|G$, and so $A^{j+1}|D$. This concludes the proof of the claim.

\bigskip

From the {\em Claim} above, we have $(\ell_1\cdots\ell_n)^r|F$, so $\deg(F)\geq rn$. But since we started with $\deg(F)=rn-2$, we then must have that $F$ is the zero polynomial.

\bigskip

For the second part, let $F\in I^{(2r)}$ of degree $\alpha\left(I^{(2r)}\right)$. Since for all $j=1,\ldots,s$, $m_j\geq 1$, we have $(2r)m_j-1\geq (2r-1)m_j$, and hence $F_x,F_y,F_z\in I^{(2r-1)}$, and so,

$$\alpha\left(I^{(2r)}\right)-1\geq \alpha\left(I^{(2r-1)}\right)\geq rn-1$$ giving the desired conclusion.

\bigskip

For the last part, if $\A$ is generic, then $n_j=2$ for all $j=1,\ldots,s$, so from property (i) above, $I$ defines a star configuration in $\mathbb P_{\mathbb K}^2$. From \cite[Lemma 2.4.1]{BH09} we then have $\alpha\left(I^{(2r)}\right)=rn$. But we also showed just above that $\alpha\left(I^{(2r)}\right)-1\geq \alpha\left(I^{(2r-1)}\right)\geq rn-1$, so when $\A$ is generic, we have also $\alpha\left(I^{(2r-1)}\right)=rn-1$.
\end{proof}

\bigskip

\begin{exm} Consider $\A\subset\mathbb P_{\mathbb{K}}^2$ with defining polynomial $xy(x-y)z$. We have
$$I:=I_3(xy(x-y)z)=\langle x,y\rangle^2\cap\langle x,z\rangle\cap\langle y,z\rangle\cap\langle x-y,z\rangle.$$ With computations done by \cite{GrSt} we have $$\alpha\left( I^{(2\cdot 2-1)}\right)=7=2\cdot 4-1.$$ So the first lower bound obtained in Proposition \ref{theorem2} is attained for this $I$ and $r=2$. This is the only example we know when this lower bound is attained for some $r\geq 2$, and when $\A$ is not generic.

About the second lower bound, we don't have an example with $\A$ not generic, $r\geq 2$, when the bound becomes an equality.
\end{exm}

\medskip

\begin{cor} \label{corollary} Denote $M:=\langle x,y,z\rangle$, the maximal irrelevant ideal of $R$. Then, for any $r\geq 1$,
\begin{itemize}
  \item[(1)] $I^{(2r-1)}\subseteq M^{r-1}\cdot I^r$.
  \item[(2)] $I^{(2r)}\subseteq M^r\cdot I^r$.
\end{itemize}
\end{cor}
\begin{proof}
From property (ii) above, ${\rm reg}(I)=n-1$. Since we just obtained that $\alpha\left(I^{(2r-1)}\right)\geq rn-1\geq r(n-1)=r{\rm reg}(I)$, by the Postulation Containment Criterion mentioned in the introduction, we have $I^{(2r-1)}\subset I^r$. So any element of $I^{(2r-1)}$ is a combination of elements of $I^r$, and is of degree $\geq rn-1$. So the polynomial coefficients of this combination must be of degree $\geq (rn-1)-r(n-1)=r-1$; hence the claim (1).

By the famous result of Ein-Lazarsfeld-Smith and Hochster-Huneke, $I^{(2r)}\subset I^r$. So, similarly, any element of $I^{(2r)}$ is a combination of elements of $I^r$, with polynomial coefficients of degree $\geq rn-r(n-1)=r$. Hence the claim (2) is shown.
\end{proof}

\medskip

\begin{rem}\label{remark2} Note that Corollary \ref{corollary} provides another evidence as to why \cite[Conjecture 2.1]{HaHu} (known as Harbourne-Huneke conjecture\footnote{The conjecture is: if $I\subset\mathbb K[x_0,\ldots,x_N]$ is a fat points ideal, then $I^{(rN)}\subseteq M^{r(N-1)}I^r$, where $M=\langle x_0,\ldots,x_N\rangle$, holds for any $r\geq 1$.}) is true, when $N=2$. In \cite[Corollary 3.9]{HaHu}, it is shown the conjecture to be true if $\A$ is generic. Our improvement is that we can take $\A$ to be any line arrangement of rank 3 in $\mathbb P_{\mathbb{K}}^2$.

Also our Proposition \ref{theorem2} and Corollary \ref{corollary} show the validity of \cite[Proposition 3.1 and Conjecture 4.1.5]{HaHu} for $N=2$ and $I$ being the ideal of our interest in this section.

Still in the planar case ($N=2$), it is worth mentioning that Harbourne-Huneke conjecture has been verified for homogeneous fat points ideals with the support generated in single degree (see \cite[Proposition 3.3]{HaHu}), as well as for a couple of other types of fat points ideals with support consisting of general points (see \cite[Theorem B]{DuSzTu}).
\end{rem}

\bigskip

\noindent
{\bf Acknowledgment.} \, The authors would like to thank the Mathematisches Forschungsinstitut Oberwolfach which funded their stay for the Mini-Workshop: Asymptotic Invariants of Homogeneous Ideals in October 2018, where the initial part of this work was developed. Moreover, we would like to thank Brain Harbourne for several helpful conversations.

\bigskip

\renewcommand{\baselinestretch}{1.0}
\small\normalsize 

\bibliographystyle{amsalpha}

\end{document}